\documentclass[12pt]{amsart}
\usepackage[T1]{fontenc}
\usepackage[english]{babel}
\usepackage[english]{babel}
\usepackage{amsmath, amssymb, amsthm, amscd,color,comment}
\usepackage{cancel}
\usepackage{cite}
\usepackage{alltt}
\usepackage[dvipsnames]{xcolor}
\usepackage{array}
\usepackage[small,bf,labelsep=period]{caption}
\usepackage{mathtools}
\usepackage{hyperref}

\usepackage{enumerate}
\newtheorem{theorem}{Theorem}[section]
\newtheorem{definition}[theorem]{Definition}
\newtheorem{example}[theorem]{Example}

\newtheorem{remark}[theorem]{Remark}
\newtheorem{lemma}[theorem]{Lemma}
\newtheorem{corollary}[theorem]{Corollary}
\newtheorem{proposition}[theorem]{Proposition}

\DeclarePairedDelimiter\ceil{\lceil}{\rceil}
\DeclarePairedDelimiter\floor{\lfloor}{\rfloor}

\def\la{\lambda}

\def\Q{{\bf Q}}

\def\N{\mathbb N}
\def\R{\mathbb R}
\def\Z{\mathbb Z}

\def\cC{\mathcal C}
\newcommand{\C}{\mathcal{C}}
\newcommand{\CL}{\mathcal{C}_\mathcal{L}}

\newcommand{\LL}{\mathcal{L}}

\def\cH{\mathcal H}

\def\cL{\mathcal L}

\def\cP{\mathcal P}

\def\cW{\mathcal W}
\def\cX{\mathcal X}
\def\cY{\mathcal Y}

\def\fqss{\mathbb F_{q^6}}
\def\fqs{\mathbb F_{q^2}}

\def\fq{\mathbb F_q}

\def\supp{{\rm Supp}}

\def\Div{{\rm Div}}
\def\dim{{\rm dim}}
\def\deg{{\rm deg}}

\def\supp{{\rm Supp}}
\def\char{{\rm Char}}

\def\char{\mbox{\rm Char}}

\def\negalpha{\text{\boldmath$\alpha$}}

\def\neg1{\text{\boldmath$1$}}

\def\negj{\text{\boldmath$j$}}
\def\negbeta{\text{\boldmath$\beta$}}

\def\neg1{\text{\boldmath$1$}}

\DeclareMathOperator{\lmd}{lmd}

\DeclareMathOperator{\Supp}{Supp}

\newcommand{\al}{\alpha}
\newcommand{\be}{\beta}

\begin{document}

\title[Non-special divisors of small degree in  Kummer extensions]{Characterization of non-special divisors of small degree on Kummer extensions and LCP codes}

\thanks{{\bf Mathematics Subject Classification (2020)}: 11T71, 14G50, 14H55}

\thanks{{\bf Keywords}: algebraic geometry codes, linear complementary pairs, Kummer extensions,
 generalized Weierstrass semigroups, non-special divisors}

\thanks{Erik Mendoza was partially supported by  Coordenação de Aperfeiçoamento de Pessoal de Nível Superior CAPES/Brasil, 001. Horacio Navarro was partially supported by Project 71411 at Universidad del Valle. 
Luciane Quoos was partially supported by Conselho
Nacional de Desenvolvimento Científico e Tecnológico CNPq/Brasil (Bolsa produtividade 307261/2023-9), Fundação de Amparo a Pesquisa do Estado do Rio de Janeiro FAPERJ/Brasil (CNE E-26/204.037/2024), and CAPES/Brasil, 001.}

\author{Erik Mendoza, Horacio Navarro, and Luciane Quoos}

\address{Instituto de Matemática, Universidade Federal do Rio de Janeiro, Cidade Universitária,
	CEP 21941-909, Rio de Janeiro, Brazil (email: erik@im.ufrj.br)}

\address{Departamento de Matemáticas, Universidad del Valle, Calle 13 \# 100-00, Cali, Colombia (email: horacio.navarro@correounivalle.edu.co)}

\address{Instituto de Matemática, Universidade Federal do Rio de Janeiro, Cidade Universitária,
	CEP 21941-909, Rio de Janeiro, Brazil (email: luciane@im.ufrj.br)}

\begin{abstract}
 A recent construction of linear complementary pairs (LCPs) of algebraic geometry codes is intimately linked to the identification of non-special divisors of small degree within a function field over a finite field. Let $\fq$ be the finite field of cardinality $q$. In this work, we consider a function field $F/\mathbb{F}_q$ of genus $g$ defined by a Kummer extension of type $y^m = f(x)$,    where $f(x)$ is a polynomial in $\fq[x]$. Based on the theory of generalized Weierstrass semigroups at several places, we provide an arithmetic criterion to identify all non-special divisors of degree $g-1$ and $g$ whose support is contained in a subset of the totally ramified places of the extension $F/\mathbb{F}_q(x)$. Furthermore, we explicitly determine all non-special divisors of degree $g-1$ in certain cases. Finally, we apply these results to provide explicit new families of LCPs algebraic geometry codes.
\end{abstract}
\maketitle

\section{Introduction}
A linear code $\C$ over a finite field $\fq$ of cardinality $q$ is a subspace of $\fq^n$, where $n\geq 1$. The code $\C$ is said to be an $[n, k, d]$ code, where $n$ is the length, $k$ is the dimension as an $\fq$-vector space, and $d$ is the (Hamming) distance. A pair of codes $(\C_1, \C_2)$ of the same length defined over $\fq$ is called a {\it linear complementary pair} (LCP) if 
$$ \C_1 + \C_2 = \fq^n \quad \text{ and } \quad \C_1 \cap \C_2 = \{0\}.$$

Linear complementary pairs of codes have attracted considerable interest due to their applications in cryptography, including fault injection attacks \cite{BCC2014} and hardening of encoded circuits against hardware Trojans \cite{NGB2014, NBDGN2015}. We refer the reader to  \cite{M1992, CMQ2025, NBDGN2015, CGOOS2018, LSL2025} for detailed characterizations and constructions of LCPs of codes.
An important family of linear codes is formed by algebraic geometry codes  (AG codes) introduced by Goppa in the eighties, see \cite{G1981}. In a seminal work  Tsfasman, Vl\u{a}du\textcommabelow{t}, and Zink \cite{TVZ1982} proved that algebraic geometry codes can exceed the Gilbert-Varshamov bound, making this family of  codes a cornerstone of coding theory. Beyond classical error correction, they have also been  utilized in distributed storage via locally recoverable codes \cite{BTV2017, TB2014, CKMTW2023},  and in secure computation via secure distributed matrix multiplication  \cite{MSH2025,LLX2026}. In this work, we investigate explicit constructions of LCPs of AG codes.

The construction of algebraic geometry codes  is based on algebraic function fields over finite fields (or equivalently, algebraic curves over finite fields). Let $F/\fq$ be a function field of genus $g$. Consider $P_1, \dots, P_n$ pairwise distinct rational places of $F/\fq$ and the divisor  $D =\sum_{i=1}^n P_i$. Fix a second divisor  $G $ with $\Supp (G) \cap \Supp (D) = \emptyset$.  The {\it algebraic geometry code}  $\C_\LL(D,G)$ is defined by 
$$\CL(D,G) = \{ (f(P_1), \cdots , f(P_n)) \, : \, f \in \LL(G) \} \subseteq \fq^n, $$
where $\LL(G)$ is the Riemann-Roch space associated with $G$ of dimension $\ell(G)$ over $\fq$.

We say a divisor $G$ in $F$ is {\it non-special } if the dimension of the Riemann-Roch space is $$\ell(G)=\deg(G)+1-g,$$ where $\deg(G)$ denotes the degree of the divisor $G$. Otherwise the divisor is said to be special. According to  the Riemann-Roch theorem,  if $\deg(G) \geq 2g - 1$, then $G$ is a non-special divisor. Conversely, the theorem also implies that if $G$ is a non-special divisor, then $\deg(G) \geq g-1$. Consequently, non-special divisors of degree $g-1$ or $g$ are referred to as non-special divisors of small degree.

Special divisors have been used to obtain an improved algorithm to decode algebraic geometry codes, see \cite{D2002}. In \cite{P1993}, Pellikaan introduced and investigated a two-variable zeta function as a generating function counting divisor classes of a given degree and dimension. On the other hand recent research on LCPs of AG codes was done by Mesnager et al.  in \cite{MTQ2017} and \cite{BDM2024}. In these works, the proposed methods for constructing LCPs of AG codes rely on the existence of non-special divisors of degree $g-1$. These results motivated the search for  such divisors explicitly. 
	
	In particular, in \cite{CLM2024} Camps Moreno et al. provide explicit constructions of non-special divisors of small degree in Kummer extensions given by $y^m = f(x)$, where $f(x) \in \mathbb{F}_q[x]$ is a separable polynomial and $(\deg (f(x)), m)=1$. In that work, the authors use an  explicit description of a generating set for the  Weierstrass semigroup at several places to construct effective non-special divisors of degree $g$. From these, they derive non-special divisors of degree $g-1$  via a result by Ballet and Le Brigand \cite{BL2006} (see Lemma \ref{lemma1}). Building on this approach, Castellanos et al. \cite{CMQ2025} exploit the divisor constructions from \cite{CLM2024}  to obtain LCPs of AG codes in these Kummer extensions.

In this paper, we consider  a Kummer extension of the form 
\begin{equation}\label{kummergeneral}
\cX: \quad y^m=f(x),
\end{equation}  where $f(x)$ in $\fq[x]$ is a  polynomial not necessarily separable.  

First, we completely characterize  all non-special divisors of degree $g-1$ and $g$  whose supports consist of  totally ramified places of $\fq(\cX)/\fq(x)$. Our approach is based on  the  explicit description of the set of absolute maximal elements for the generalized Weierstrass semigroup of the extension \eqref{kummergeneral},  as provided  in  \cite{CMT2026}. By combining this with  results from \cite{MTT2019} and \cite{CMS2025} (see Theorem \ref{teo_Riemann_Roch_dimension} and Corollary \ref{coro1}, resp.), we obtain a formula for the dimension of certain Riemann-Roch spaces, see Theorem \ref{teo_dimension_RRS_Kummer}.
A key advantage of this approach, compared to using the standard Weierstrass semigroups, is that it facilitates the direct study of divisors of degree $g-1$ and $g$ that are not necessarily effective. In Theorems \ref{teo_g-1_Kummer} and \ref{carac-grad-g-1} an explicit description of non-specials divisors of degree $g-1$ for certain Kummer extensions can be found.

Second, we apply these newly obtained  non-special divisors of small degree to derive new families of LCPs AG codes over Kummer extensions. In doing so,  we generalize the constructions of non-special divisors of small degree and LCPs of AG codes previously established in \cite{CLM2024}  and   \cite{CMQ2025}, respectively.

This paper is organized as follows. Section~$2$ reviews preliminaries on function fields, generalized Weierstrass semigroups, Kummer extensions, AG codes, and LCPs of AG codes. Section~$3$ presents a formula for the dimension of the Riemann-Roch space of a divisor supported on a set of totally ramified places of the Kummer extension~\eqref{kummergeneral}; see Theorem~\ref{teo_dimension_RRS_Kummer}. This formula is then used to establish arithmetic criteria to determine whether such divisors are non-special of degrees $g-1$ and $g$; see Theorems~\ref{teo-car-div-deg-g-1} and~\ref{teo2-car-div-deg-g}, respectively. Based on these results, Section~$4$ gives an explicit description of all non-special divisors of degree $g-1$ supported on a set of totally ramified places, under specific conditions on the multiplicities of $f(x)$. Finally, Section~$5$ presents new and generalized constructions of LCPs of AG codes, see Theorems \ref{Teocodes1}, \ref{Teocodes2} and \ref{TeocodesR}, and some explicit examples.

\section{Preliminaries and Notation}

Throughout this article, we let $q$ be the power of a prime $p$ and $\fq$ be the finite field with $q$ elements. For $a$ and $b$ integers, we denote by $(a, b)$ the greatest common divisor of $a$ and $b$, and by $b \bmod{a} $ the smallest non-negative integer congruent with ($b$ modulo $a$). For $c\in \R$, we denote by $\floor*{c}$ and $\ceil*{c}$  the floor and  ceiling functions of $c$, respectively. We also denote by $\N$ the set of non-negative integers. 


\subsection{Functions fields and generalized Weierstrass semigroups}

Let  $F/\fq$ be a function field in one variable of genus $g$ and constant field $\fq$. We denote by $\mathcal P_{F}$ the set of places in $F$, and by $\nu_{P}$ the discrete valuation of $F/\fq$ associated to the place $P\in \cP_{F}$. We say a  place is rational if $\deg (P)=1$.  Let  $\Div (F)$ be the free abelian group generated by the places in  $\mathcal P_{F}$. An element $G=\sum_{i=1}^s n_iP_i$ in $\Div (F)$ is a divisor and its degree is given by $\deg(G)=\sum_{i=1}^s n_i\deg (P_i)$. For a function $z \in F$ we let $(z), (z)_\infty,$ and $(z)_0$ stand for the principal, pole, and zero divisors of the function $z$, respectively. The Riemann-Roch space associated to the divisor $G\in \Div(F)$ of $F/\fq$  is defined by
$$
\cL(G)=\{z\in F: (z)+G\geq 0\}\cup \{0\},
$$  
and we denote by $\ell(G)$ its dimension as a vector space over $\fq$. From Riemann Theorem we have $\ell(G)\geq \deg (G)  +1-g.$
A divisor $G$ is called a {\it non-special divisor} if $$\ell(G)=\deg (G) +1-g;$$ otherwise, $G$ is called {\it special}.  

The existence and characterization of non-special divisors are fundamental to construct LCPs codes from algebraic geometry codes. A connection between non-special divisors of degrees $g$ and $g-1$ can be derived from the following results. 
	\begin{lemma}\cite[Lemma 3]{BL2006}\label{lemma1}
		If $G$ is an effective non-special divisor of degree $g$  on a function field $F$ 
		and  $P$ is a rational place such that $P \not\in \supp(G)$, then $G-P$ is non-special (of degree $g-1$).
	\end{lemma}   
	
	\begin{lemma}\cite[Remark 1.6.11]{S2009}   	
		If $G$ is a non-special divisor of degree $g-1$  on a function field $F$ 
		and $P$ is a rational place, then $G+P$ is non-special (of degree $g$).	
	\end{lemma}
    We observe that both hypotheses in Lemma \ref{lemma1}, $G$ is effective and that $P \not\in \supp(G)$, are necessary. In fact, let $\cX$ be an hyperelliptic curve of genus $g$ and $P$ be a rational place on $\cX$ with Weierstrass semigroup $H(P)=\{0, g+1, g+2, g+3, \dots \}$. Then $\ell(gP)=\ell((g-1)P)=1$, so $gP$ is an effective non-special divisor, but $(g-1)P$ is special. Now consider the curve $y^4=x^3+x$ over $\mathbb{F}_9$ of genus $3$, let $Q_\infty$ be the only place at infinity, and $Q_1, Q_2$ and $Q_3$ be the only places in $\mathbb{F}_9(x,y)$ over the zeros of $x^3+x$. Then  $G=-2Q_1+2Q_2+3Q_3$ is a non-special (non-effective) divisor of degree $3$, but $G-Q_\infty$ is a special divisor.
	
 In this work, we use the theory of generalized Weierstrass semigroups at several places and Riemann-Roch spaces to characterize non-special divisors of degree $g$ and $g-1$ in Kummer extensions. Let $\Q=(Q_1, \dots, Q_n)$ be an $n$-tuple of pairwise distinct rational places in $F$. The Weierstrass semigroup $H(\Q)$ of $F$ at $\Q$ is defined by
$$
H(\mathbf{Q}) := \left\{(a_{1}, \ldots, a_{n}) \in \mathbb{N}^ {n} :  (z)_{\infty} = \textstyle\sum_{i=1}^ {n} a_{i}Q_{i} \text{ for some }z\in F \right\}.
$$
We notice that for $n=1$ ones retrieve the classical Weierstrass semigroup in one rational place. These semigroups have been intensively used on the construction of algebraic geometry codes, see for instance, \cite{GL92, GKL93, CT05, M05, HY20, CMQ2025 }.  The elements in the finite complement ${G(\mathbf{Q})=\N^n\setminus H(\mathbf{Q})}$ are called \emph{gaps} of $F$ at $\mathbf{Q}$ and a characterization can be fully articulated in terms of the dimensions of certain associated Riemann-Roch spaces. In fact,  given an $n$-tuple $\negalpha=(\al_1, \dots, \al_n)\in \N^n$ consider the divisor defined by 
 \begin{equation}\label{defDalpha}
 D_\negalpha(\Q) = \alpha_1 Q_1 + \cdots + \alpha_n Q_n.
 \end{equation}
  Then $\negalpha$ is an element in $G(\Q)$ if and only if
 $$\ell(D_{\negalpha}(\Q)) = \ell(D_{\negalpha}(\Q) - Q_i) \text{ for some } 1 \leq i \leq n.$$

The notion of Weierstrass semigroups at several rational places was extended to $n$-tuples in $\mathbb{Z}^n$ by considering the ring $R_{\mathbf{Q}}$ of functions on $F$ that are regular outside the finite set of places ${Q_1, \dots, Q_n}$. This subject was investigated by Delgado in  \cite{D1990} over algebraically closed fields and, by Beelen and Tuta\c{s} over finite fields (see \cite{BT2006}). The generalized Weierstrass semigroup $\widehat{H}(\Q)$ of $F$ at $\Q$ is given by
$$
\widehat{H}(\Q):=\{(-\nu_{Q_1}(z),\dots ,-\nu_{Q_n}(z))\in \Z^n : z \in R_\Q\setminus\{0\}\}.
$$
The
advantage of this approach over the Weierstrass semigroup in several places is that it will allow us to study non-effective
divisors as well, a key point for the classification of non-special divisors of degree $g-1$. 

 Provided that $q \geq n$, we have that the Weierstrass semigroup of $F$ at $\Q$ can be obtained by the relation $H(\Q)=\widehat{H}(\Q)\cap \N^n$, see \cite[Proposition 2]{BT2006}.

An important concept in the study of $\widehat{H}(\Q)$,
is that of absolute maximal elements. 
These elements play a crucial role in describing the generalized Weierstrass semigroup $\widehat{H}(\Q)$, see \cite[Theorem 3.4]{MTT2019}. To introduce this concept, we need first to introduce some notation. Let $I:=\{1,\ldots,n\}$. For $i\in I$, a nonempty subset $J\subsetneq I$, and $\negalpha=(\alpha_1,\ldots,\alpha_n)\in \Z^n$,   denote

\begin{itemize}
	\item [$\bullet$] $\overline{\nabla}_J (\negalpha):=\{\negbeta\in \Z^n : \be_j=\al_j \text{ for }j\in J \text{ and } \be_i < \al_i \text{ for }i\notin J\}$,
	\vspace{0.2cm}
	\item [$\bullet$] $\nabla_J(\negalpha):=\overline{\nabla}_J(\negalpha)\cap \widehat{H}(\Q)$,
	\vspace{0.2cm}
	\item [$\bullet$] $\overline{\nabla}(\negalpha):=\cup_{i=1}^{n}\overline{\nabla}_i(\negalpha)$, where $\overline{\nabla}_i(\negalpha):= \overline{\nabla}_{\{i\}}(\negalpha)$, and\vspace{0.2cm}
	\item [$\bullet$] $\nabla(\negalpha):=\overline{\nabla}(\negalpha)\cap \widehat{H}(\Q)$.
\end{itemize}

\begin{definition} \label{defi maximals}
	An element $\negalpha\in \widehat{H}(\Q)$ is called maximal if $\nabla(\negalpha)=\emptyset$. If moreover $\nabla_J(\negalpha)=\emptyset$ for every $J\subsetneq I$ with $|J| \geq 2$, we say that $\negalpha$ is absolute maximal. The set of absolute maximal elements in $\widehat{H}(\Q)$ is  denoted by $\widehat{\Gamma}(\Q)$.
\end{definition}

The absolute maximal elements in $\widehat{H}(\Q)$ are closely related to the dimension of the  Riemann-Roch spaces $\ell(D_{\negalpha}(\Q))$ as
established below. For that purpose at first we introduce the following  partial order in $\Z^n$, for $\negbeta $ and $ \negalpha $ in $\Z^n$ we say that
$$\negbeta \leq \negalpha  \Leftrightarrow \be_i\leq\al_i \text{ for all } i=1, \dots, n.$$ 
For $\negalpha\in \Z^n$, define
$$
\widehat{\Gamma}_{\Q}(\negalpha):=\{\negbeta\in \widehat{\Gamma}(\Q): \negbeta \leq \negalpha\}.
$$
Note that $\widehat{\Gamma}_{\Q}(\negalpha)$ is a finite subset of
$\Z^n$. For $i\in\{1, \dots, n\}$ and $\negbeta, \negbeta' \in\widehat{\Gamma}_{\Q}(\negalpha)$, define
$$
\negbeta \equiv_i \negbeta'\quad\text{if and only if}\quad \be_i=\be_i'.
$$
A straightforward verification shows that $\equiv_i$ is an equivalence relation in $\widehat{\Gamma}_{\Q}(\negalpha)$. We denoted by $[\negbeta]_i$  the equivalence classes for $\negbeta\in \widehat{\Gamma}_{\Q}(\negalpha)$ in the
quotient set $\widehat{\Gamma}_{\Q}(\negalpha)/\equiv_i$. The following theorem provides the dimension of the Riemann–Roch space associated to divisor $D_{\negalpha}(\Q)$ in terms of $\widehat{\Gamma}_{\Q}(\negalpha)$.

\begin{theorem}\cite[Theorem 3.5]{MTT2019}\label{teo_Riemann_Roch_dimension}
	Let $\negalpha\in \Z^n$, $i\in \{1, \dots, n\}$, and assume that $q\geq n$. Then
	$$
	\ell(D_{\negalpha}(\Q))=|\widehat{\Gamma}_{\Q}(\negalpha)/\equiv_i|.
	$$
\end{theorem}

\subsection{Kummer extensions}
Let $m\geq 2$ be a integer such that $\char(\fq) \nmid m$, and $f(x)\in \fq[x]$ be a polynomial with $\deg(f)\geq 2$  that is not a $d$-th power of an element in $\fq(x)$, where $2 \leq d$ divides $m$. Assume that all roots of $f(x)$ are in $\fq$. Consider the algebraic curve over $\fq$ defined by the affine equation
\begin{equation}\label{equationX}
\cX: \quad y^m=f(x)
\end{equation}
with function field $\fq(\cX)$ of genus $g(\cX) \geq 1$. Then $\fq(\cX)/\fq(x)$ is a Kummer extension. We fix the following notation for this Kummer extension:

\begin{itemize}
\item $P_1, P_2, \dots, P_r\in \cP_{\fq(x)}$ are all the places corresponding to the zeros and the pole of $f(x)$ in $\fq(x)$ (notice that $r\leq \deg(f)+1$), 
\vspace{0.2cm}
\item $\la_{k}:=\nu_{P_k}(f(x))$ is the multiplicity of $P_k$ in $f(x)$, where $\nu_{P_k}$ is the valuation associated to the place $P_k$,
\item Suppose $n \leq r$ and $P_1, P_2, \dots, P_n\in \cP_{\fq(x)}$ are pairwise distinct  totally ramified places in the extension $\fq(\cX)/\fq(x)$. For $k=1, \dots, n$ we denote by $Q_k\in \cP_{\fq(\cX)}$  the unique extension in $\fq(\cX)$ over $P_k$,   and by $\Q:=(Q_1, Q_2, \dots, Q_n)$. Note that in case the only pole of $f(x)$ is totally ramified in $\fq(\cX)/\fq(x)$ it can also appear in the $n$-tuple $\Q=(Q_1, Q_2, \dots, Q_n)$.
\end{itemize}

In \cite{CMS2025}, the authors provided an explicit and compact description of the gap set $G(Q)$ at any totally ramified place $Q\in \cP_{\fq(\cX)}$ in the extension $\fq(\cX)/\fq(x)$. To achieve this, for each $1 \leq i \leq m-1$ and $1 \leq s \leq n$, they defined the following functions:
\begin{equation}\label{funcoestsbi}
t_s(i):= (i\la_s)\bmod m\quad \text{ and }\quad \be(i):=\sum_{k=1}^{r}\ceil*{\frac{i\la_k}{m}}-1.
\end{equation}

The Weierstrass Gap Theorem states that the number of gaps in the semigroup of a rational place is precisely the genus of the curve. As a consequence, they have the following result.

\begin{corollary}\label{coro1}\cite[Corollary 3.3]{CMS2025}
Suppose $\cX$ is as in (\ref{equationX}) and that $\fq(\cX)/\fq(x)$ has  at least one totally ramified place. Then \begin{equation}\label{erikethaneq} \sum_{i=1}^{m-1}\be(i)=g(\cX),\end{equation} where  $g(\cX)$ is the genus of the function field $\fq(\cX)$.
\end{corollary}

On the other hand, in \cite{CMT2026} the authors provided an explicit description of the set of absolute maximal elements $\widehat{\Gamma}(\Q)$ of the generalized Weierstrass semigroup $\widehat{H}(\Q)$.

\begin{theorem}\label{teo_widehat_Gamma}\cite[Theorem 3.4]{CMT2026}
	For a Kummer extension as in (\ref{equationX}), let $2\leq n \leq q$ and $\Q=(Q_{1}, Q_{2}, \dots, Q_{n})$. Then
	\begin{align*}
		\widehat{\Gamma}(\Q)=&\Bigg\{(mj_1+t_{1}(i), \dots,  mj_n+t_{n}(i)):
		\begin{array}{l}
			1\leq i \leq m-1,\,\, j_1, \dots, j_n\in \Z, \\
			j_1+\cdots +j_n=\beta(i)+1-n
		\end{array}
		\Bigg\}\\
		& \bigcup \Bigg\{(mj_1, mj_2, \dots, mj_n):\, j_1, j_2, \dots, j_n \in \Z, \, \,   j_1+j_2+\cdots+j_n=0 \Bigg\}.
	\end{align*}
\end{theorem}

	\subsection{ Algebraic geometry codes and LCPs of algebraic geometry codes}
Let $N \geq 1$ be an integer. A {\it linear code} $\C$ over $\fq$ is an $\fq$-subspace of $\fq^N$. Associated to a linear code, we have three parameters: the {\it length} $N$, the {\it dimension} $k$ of $\C$ as an $\fq$-vector space, and the {\it minimum distance} (Hamming distance)
$$d= \min \{ \text{wt}(\textbf{x})  : \textbf{x} \in \C \setminus \{ 0\} \},$$ 
where $\text{wt}(\textbf{x}) = \# \{ i : x_i \neq 0 \}$ for any  $N$-tuple $\textbf{x}=(x_1, \dots, x_N) \in \fq^N$. We denote such code $\cC$ as an $[N,k,d]$, or just $[N,k]$ code.

\begin{definition}A pair of linear codes $\C_1$ and $\C_2$ in $\fq^N$ is said to be a {\it linear complementary pair} (LCP for short) if $\C_1 \oplus \C_2 = \fq^N.$ 
\end{definition}

We now recall the construction of linear algebraic geometry codes. Let $F/\fq$ be a function field of genus $g$. Consider $P_1, \dots, P_N \in \mathcal P_{F}$ pairwise distinct rational places of $F/\fq$ and denote $D :=\sum_{i=1}^N P_i$. Given a divisor $G \in \Div(F)$ with $\Supp (G) \cap \Supp (D) = \emptyset$,  the {\it algebraic geometry code} (AG code for short) $\C_\LL(D,G)$ is defined by 
$$\CL(D,G) := \{ (f(P_1), \cdots , f(P_N)) \, : \, f \in \LL(G) \} \subseteq \fq^N. $$

The parameters of AG codes satisfy the following well-known bounds.	
\begin{proposition}\cite[Theorem 2.2.2]{S2009}\label{parameters} The code $\CL(D, G)$ has parameters $[N, k, d]$ which satisfy $$k=\ell(G)-\ell(G-D) \quad \mbox{and} \quad d \ge N - \deg (G).$$ In particular, for $2g - 2 < \deg (G) < N$ it holds that $k=\ell(G)=\deg(G) + 1 - g$.  
\end{proposition}

For two divisors $A, B$ in $\Div(F)$, we define the {\it greatest common divisor} of $A$ and $B$ by
$$\gcd(A,B) := \sum_{P \in \mathcal{P}_F }\min\{ \nu_P(A), \nu_P(B) \} P,$$
and  the {\it least multiple divisor} of $A$ and $B$  by
$$\lmd(A,B) := \sum_{P \in \mathcal{P}_{F}} \max\{ \nu_P(A), \nu_P(B) \} P.$$
The $\gcd(A,B)$ and $\lmd(A,B)$ fulfill these two fundamental properties:
\begin{equation}\label{basicP}
	\LL(A) \cap \LL(B) = \LL(\gcd(A,B)) \quad \mbox{and} \quad \gcd(A,B)+\lmd(A,B) = A+B.	
\end{equation}

We know introduce the notion of LCP codes in the context of AG  codes.
Let $P_1, \dots, P_N$ be pairwise distinct rational places of $F$ and consider the divisor $D = \sum_{i=1}^N P_i$. Consider $G$ and $H$  two divisors of $F$ such that $$\Supp(G) \cap \Supp(D) = \Supp(H) \cap \Supp(D) = \emptyset.$$ 
The pair $(\CL(D,G), \CL(D, H))$ of AG codes is an LCP of AG codes over $\fq$ if 
$$\CL(D,G) \oplus \CL(D, H) = \fq^N.$$ 
In other words, a pair $(\CL(D,G), \, \CL(D, H))$ of AG codes is an LCP if and only if 
\begin{equation}\label{LCPcond}
	\dim(\CL(D,G)) + \dim(\CL(D,H)) = N \quad \text{ and } \quad \CL(D,G) \cap \CL(D,H) = \{ 0 \}.
\end{equation}

In \cite{BDM2024}, Bhowmick, Dalai, and  Mesnager provided some  conditions on the divisors $G$ and $H$ to obtain a pair of LCP of AG codes. We use Riemann-Roch Theorem and state the theorem as follows.
\begin{theorem}\cite[Theorem 3.5]{BDM2024}\label{thmlcp3.5}. 
	Let $\CL(D,G)$ and $\CL(D,H)$ be two algebraic geometry codes over a function field $F/\fq$ of genus $g \neq 0$. Suppose $D$ has degree $N$ and the divisors $G$ and $H$ satisfy $2g-2 < \deg(G), \, \deg(H) < N$ and   
	\begin{enumerate}[i)]
		\item  $\deg(G)+\deg(H)=N+2g-2$,
		\item $\deg(\gcd(G, H))=g-1$, and 
		\item both divisors $\gcd(G,H)$ and $ \lmd(G,H)- D$ are non-special.
	\end{enumerate}
	Then the pair $(\CL(D,G), \, \CL(D,H))$ is LCP.
\end{theorem}

\section{An arithmetic characterization of non-special divisors}

In this section, considering a Kummer extension $\fq(\cX)/\fq(x)$ as in (\ref{equationX}), we provide an arithmetic criterion to determine all non-special divisors of degree $g(\cX)$ and $g(\cX)-1$ whose support is contained in a set of totally ramified places $\{Q_1, Q_2, \dots, Q_n\}$.

We begin by providing a formula for the dimension of the Riemann-Roch space associated to any divisor with support contained in $\{Q_1, Q_2, \dots, Q_n\}$.

\begin{theorem}\label{teo_dimension_RRS_Kummer}
	Let $\negalpha=(\al_1, \al_2, \dots, \al_n)\in \Z^n$ and assume that $2\leq n\leq q$. Consider a Kummer extension as in (\ref{equationX}), $\{Q_1, \dots, Q_n\}$ a set of totally ramified places, and the divisor $D_{\negalpha}(\Q) =\sum_{i=1}^n \alpha_i Q_i$. Then the dimension of the Riemann-Roch space $\cL(D_{\negalpha}(\Q))$ is given by 
	$$
	\ell(D_{\negalpha}(\Q))=\max\left\{0, 1+\sum_{k=1}^{n}\floor*{\frac{\al_k}{m}}\right\}+\sum_{i=1}^{m-1}\max\left\{0, n-\be(i)+\sum_{k=1}^{n}\floor*{\frac{\al_k-t_k(i)}{m}}\right\},
	$$
	 where $ t_k(i)= (i\la_k)\bmod m$ and $\be(i)=\sum_{k=1}^{r}\ceil*{\frac{i\la_k}{m}}-1,$ as in (\ref{funcoestsbi}).
\end{theorem}
\begin{proof}
	As an application of  Theorem \ref{teo_Riemann_Roch_dimension}, choosing $i=1$ we have $$\ell(D_{\negalpha}(\Q))=|\widehat{\Gamma}_{\Q}(\negalpha)/\equiv_1|.$$ Therefore, it is sufficient to compute the following cardinality: 
	$$|\{\be_1:\negbeta=(\be_1, \dots, \be_n) \in\widehat{\Gamma}_{\Q}(\negalpha)\}|.$$
	Using the notation as in (\ref{funcoestsbi}), Theorem \ref{teo_widehat_Gamma} provides a decomposition of  $\widehat{\Gamma}(\Q)$ as a disjoint union  $\widehat{\Gamma}(\Q)=\cup_{i=0}^{m-1}\widehat{\Gamma}_i$, where 
	\begin{align*}
	\bullet \, \widehat{\Gamma}_0&=\{(mj_1, \dots, mj_n)\,: \, j_1, \dots, j_n\in \Z \text{ and } j_1+\cdots+j_n=0\}, \text{ and} \\
	\bullet \, \widehat{\Gamma}_i&=\{(mj_1+t_1(i), \dots, mj_n+t_n(i))\,: \,  j_1, \dots, j_n\in \Z \text{ and }  j_1+\cdots+j_n=\be(i)+1-n\}, 
	\end{align*}
	 for $ i=1, \dots,  m-1.$ 
	  
	 Thus, 
	$$|\{\be_1:\negbeta=(\be_1, \dots, \be_n) \in\widehat{\Gamma}_{\Q}(\negalpha)\}|=\sum_{i=0}^{m-1}|\{\be_1:\negbeta=(\be_1, \dots, \be_n) \in\widehat{\Gamma}_{\Q}(\negalpha)\cap \widehat{\Gamma}_i\}|.$$
We now consider two cases.

\vspace{0.2cm}

	$\bullet$ If $\negbeta\in \widehat{\Gamma}_{\Q}(\negalpha)\cap \widehat{\Gamma}_0$, then $\negbeta=(mj_1, mj_2, \dots, mj_n)$, where $j_1+\cdots+j_n=0$, and $mj_k\leq \al_k$ for every $1\leq k\leq n$. Thus,
	$$
	-\sum_{k=2}^{n}\floor*{\frac{\al_k}{m}}\leq j_1\leq \floor*{\frac{\al_1}{m}},
	$$
	and we obtain
	$$
	|\{\be_1:\negbeta=(\be_1, \dots, \be_n) \in\widehat{\Gamma}_{\Q}(\negalpha)\cap \widehat{\Gamma}_0\}|=\max\left\{0, 1+\sum_{k=1}^{n}\floor*{\frac{\al_k}{m}}\right\}.
	$$
	
	$\bullet$ If $\negbeta\in \widehat{\Gamma}_{\Q}(\negalpha)\cap \widehat{\Gamma}_i$ for some $1\leq i \leq m-1$, then $\negbeta=(mj_1+t_1(i),\dots, mj_n+t_n(i))$, where $j_1+\cdots+j_n=\be(i)+1-n$, and $mj_k+t_k(i)\leq \al_k$ for every $1\leq k\leq n$. Thus, we have
	$$
	\be(i)+1-n-\sum_{k=2}^{n}\floor*{\frac{\al_k-t_k(i)}{m}}\leq j_1\leq \floor*{\frac{\al_1-t_1(i)}{m}}.
	$$
	So,
	$$
	|\{\be_1:\negbeta=(\be_1, \dots, \be_n) \in\widehat{\Gamma}_{\Q}(\negalpha)\cap \widehat{\Gamma}_i\}|=\max\left\{0, n-\be(i)+\sum_{k=1}^{n}\floor*{\frac{\al_k-t_k(i)}{m}}\right\},
	$$
	and the result follows.
\end{proof}

In the following theorem, we characterize all non-special divisors of degree $g(\cX)-1$ with support contained in totally ramified places $\{Q_1, Q_2, \dots, Q_n\}$ in the Kummer extension $\fq(\cX)/\fq(x)$.  This will be used in the next section to present all explicit non-special divisors of degree $g(\cX)-1$ in some cases.

\begin{theorem}\label{teo-car-div-deg-g-1}
	Following the same notation as in Theorem \ref{teo_dimension_RRS_Kummer}.	
	\begin{enumerate}[i)]	
	\item  Let $\negalpha^\prime=(\al_1^\prime,  \dots, \al_n^\prime) \in \Z^n$ and assume that $2\leq n \leq q$. 
 Then  $D_{\negalpha^\prime}(\Q)$ is a non-special divisor of degree $g(\cX)-1$ if and only if 
\begin{equation}\label{primer_criterio_g-1}
\sum_{k=1}^{n}\floor*{\frac{\al_k^\prime}{m}}=-1 \quad \text{and}\quad n+\sum_{k=1}^{n}\floor*{\frac{\al_k^\prime-t_k(i)}{m}}=\be(i)\text{ for every }1\leq i \leq m-1.
\end{equation}

\item Dividing $\alpha_i^\prime $ by $m$, we can write	$$\label{divD}
	D_{\negalpha^\prime}(\Q)=\sum_{k=1}^{n}(mj_k+\al_k)Q_k, 
	\text{ where }  j_k \in \Z \text { and } 0\leq \al_k\leq m-1 \text{  for  }  k= 1, \dots  n.
	$$
	Then $D_{\negalpha^\prime}(\Q)$ is a non-special divisor of degree $g(\cX)-1$ if and only 
	if 
	\begin{equation}\label{eq_criterio_g-1}
		\sum_{k=1}^{n}j_k=-1 \quad\text{and}\quad \be(i)=n-1+\sum_{k=1}^{n}\floor*{\frac{\al_k-t_k(i)}{m}} \text{ for every } 1\leq i \leq m-1.
	\end{equation}
\end{enumerate}
\end{theorem}
\begin{proof}
	The divisor $D_{\negalpha^\prime}(\Q)$ is a non-special divisor of degree $g(\cX)-1$ if and only if $\ell(D_{\negalpha^\prime}(\Q))=0$ and $\sum_{k=1}^{n}\al_k^\prime =g(\cX)-1$. From Theorem \ref{teo_dimension_RRS_Kummer}, $\ell(D_{\negalpha^\prime}(\Q))=0$ if and only if 
	$$
	\sum_{k=1}^{n}\floor*{\frac{\al_k^\prime}{m}}+1\leq 0 \quad \text{and}\quad n+\sum_{k=1}^{n}\floor*{\frac{\al_k^\prime-t_k(i)}{m}}\leq\be(i)\text{ for every }1\leq i \leq m-1.
	$$
	Thus, from Corollary \ref{coro1}, we obtain 
	\begin{align*}
		g(\cX)&=\sum_{i=1}^{m-1}\be(i)
		\geq\sum_{i=1}^{m-1}\left(n+\sum_{k=1}^{n}\floor*{\frac{\al_k^\prime-t_k(i)}{m}}\right)\\
		&=n(m-1)+\sum_{k=1}^{n}\sum_{i=1}^{m-1}\floor*{\frac{\al_k^\prime-t_k(i)}{m}}
		=n(m-1)+\sum_{k=1}^{n}\sum_{i=1}^{m-1}\floor*{\frac{\al_k^\prime-i}{m}}\\
		&=n(m-1)+\sum_{k=1}^{n}\left(\al_k^\prime-\floor*{\frac{\al_k^\prime}{m}}-m+1\right)
		=g(\cX)-1-\sum_{k=1}^{n}\floor*{\frac{\al_k^\prime}{m}}\\
		&\geq g(\cX).
	\end{align*}
	The result of the first item of the theorem follows. 
	
	Now we prove item $ii)$. If $D_{\negalpha^\prime}(\Q)$ as in item $ii)$ satisfies condition (\ref{eq_criterio_g-1}), then it is clear it satisfies the conditions in $i)$, and we conclude it is non-special. Conversely, given  $$D_{\negalpha^\prime}(\Q)=\sum_{k=1}^{n}(mj_k+\al_k)Q_k \text{  where } j_k \in \Z \text{ and } 0\leq \al_k\leq m-1 \text{ for } k= 1, \dots,  n$$ a non-special divisor of degree $g(\cX)-1$,  then  $D_{\negalpha^\prime}(\Q)$ satisfies the conditions given in item $i)$. That is,    
	$$-1=\sum_{k=1}^{n}\floor*{\frac{mj_k+\al_k}{m}}=\sum_{k=1}^{n}\left(j_k+\floor*{\frac{\al_k}{m}}\right)=\sum_{k=1}^{n}j_k,$$
	and for every $1\leq i \leq m-1$,
	\begin{align*}
		\be(i)&=n+\sum_{k=1}^{n}\floor*{\frac{mj_k+\al_k-t_k(i)}{m}}=n+\sum_{k=1}^{n}\left(j_k+\floor*{\frac{\al_k-t_k(i)}{m}}\right)\\&=n-1+\sum_{k=1}^{n}\floor*{\frac{\al_k-t_k(i)}{m}}.	
	\end{align*}
	
%

\end{proof}

\begin{remark}\label{remark_equiv}
Using the same   notation as in Theorem \ref{teo-car-div-deg-g-1},  note that if $Q$ is any totally ramified place in $\fq(\cX)/\fq(x)$ then $mQ_i\sim mQ$ for every $1\leq i \leq n$. Thus, if $j_1, j_2, \dots, j_n\in \Z$ are such that $\sum_{k=1}^{n}j_k=-1$, we obtain
$$
\sum_{k=1}^{n}mj_kQ_k\sim -mQ
$$
for any totally ramified place $Q$ in $\fq(\cX)/\fq(x)$. So, from Theorem \ref{teo-car-div-deg-g-1}, we conclude that $D\in \Div(\cX)$ with $\supp(D)\subseteq \{Q_1, Q_2, \dots, Q_n\}$ is a non-special of degree $g(\cX)-1$ if and only if 
$$
D\sim -mQ+\sum_{k=1}^{n}\al_kQ_k,
$$
where $0\leq \al_k\leq m-1$ for $k=1, \dots, n$ and the function $\be(i)$ as in (\ref{funcoestsbi}) satisfies
$$
\be(i)=n-1+\sum_{k=1}^{n}\floor*{\frac{\al_k-t_k(i)}{m}} \text{ for every } 1\leq i \leq m-1.
$$
\end{remark}

Next corollary shows, in particular, that not for all Kummer extensions we have non-special divisors of small degree with support only in the totally ramified places.

\begin{corollary}
Following the same notation as in Theorem \ref{teo-car-div-deg-g-1}, suppose that the multiplicity of all roots of $f(x)$ are smaller than $m$. If there exists a non-special divisor of degree $g(\cX)-1$ with support contained in $\{Q_1, Q_2, \dots, Q_n\}$, then
\begin{equation}\label{condition_g-1}
	\max\{\be(i): 1\leq i \leq m-1\} \leq n-1.
\end{equation}
In particular, $ \floor{\frac{\deg f}{m}} \geq r-n-1$.
\end{corollary}
\begin{proof}
If there exists a non-special divisor of degree $g(\cX)-1$ with support contained in $\{Q_1, Q_2, \dots, Q_n\}$, then there are $\al_1, \al_2, \dots, \al_n$, where $0\leq \al_k\leq m-1$ for every $1\leq k \leq n$, satisfying the second equality of condition (\ref{eq_criterio_g-1}). Since $
\floor*{\frac{\al_k-t_k(i)}{m}}\leq 0,$ we obtain $\beta(i)\leq n-1$ for $1\leq i \leq m-1$. Thus, we deduce  
that 
\begin{equation*}
	\max\{\be(i): 1\leq i \leq m-1\} \leq n-1.
\end{equation*}
In particular, computing $\be(1)$, we obtain $ \floor{\frac{\deg f}{m}} \geq r-n-1$.
\end{proof}

\begin{example} 
Consider the curve introduced by Giulietti and Korchm\'{a}ros \cite{GK2009} given by the affine equation
\begin{equation}\label{GK}
	\mathcal{GK}:\, y^{q^3+1}=(x^q+x)h(x)^{q+1}, \quad \text{where}\quad h(x)=\sum_{i=0}^{q}(-1)^{i+1}x^{i(q-1)}.
\end{equation}
This curve is maximal over $\fqss$ and is the first example of a maximal curve not covered by the Hermitian curve over $\fqss$. 
For this curve we have $r=q^2+1$, $n=q+1$, $\deg f=q^3$, and $m=q^3+1$. Hence,  $$ \floor*{\frac{\deg f}{m}}=\floor*{\frac{q^3}{q^3+1}}=0 < q^2-q-1= r-n-1.$$
We conclude that, considering the affine equation in (\ref{GK}) of the $\mathcal{GK}$ curve, there are no non-special divisors of degree $g(\mathcal{GK})-1$ with support in the set of totally ramified places of $\fqss(\mathcal{GK})/\fqss(x)$.  
\end{example}

Analogously to Theorem \ref{teo-car-div-deg-g-1}, in the following results we characterize all non-special divisors of degree $g(\cX)$ with support contained in $\{Q_1, Q_2, \dots, Q_n\}$.

\begin{theorem}\label{teo-car-div-deg-g}
	Following the same notation as in Theorem \ref{teo_dimension_RRS_Kummer}, let $\negalpha=(\al_1, \al_2, \dots, \al_n)\in \Z^n$ and assume that $2\leq n \leq q$. Then  $D_{\negalpha}(\Q)$ is a non-special divisor of degree $g(\cX)$ if and only if one of the following
	conditions is satisfied:
	\begin{enumerate}[i)]
	\item \label{cond.i-teo-car-div}  $
	\displaystyle\sum_{k=1}^{n}\floor*{\frac{\al_k}{m}}=0$ and $n+\displaystyle\sum_{k=1}^{n}\floor*{\frac{\al_k-t_k(i)}{m}}=\be(i)$ for every $1\leq i \leq m-1$, or
	\item \label{cond.ii-teo-car-div} $
	\displaystyle\sum_{k=1}^{n}\floor*{\frac{\al_k}{m}}=-1$, $n-1+\displaystyle\sum_{k=1}^{n}\floor*{\frac{\al_k-t_k(j)}{m}}=\be(j)$ for some $1\leq j \leq m-1$, and $ n+\displaystyle\sum_{k=1}^{n}\floor*{\frac{\al_k-t_k(i)}{m}}=\be(i)$ for every $1\leq i \leq m-1$ with $i\neq j$.
\end{enumerate}
\end{theorem}

\begin{proof}
The  divisor $D_{\negalpha}(\Q)$ is a non-special divisor of degree $g(\cX)$ if and only if 
$$\ell(D_{\negalpha}(\Q))=1 \quad \text{and} \quad \sum_{k=1}^{n}\al_k=g(\cX).$$ We recall that from Corollary \ref{coro1} we also have $g(\cX)=\sum_{i=1}^{m-1}\be(i)$.

Now, one can deduce from Theorem \ref{teo_dimension_RRS_Kummer}, that $\ell(D_{\negalpha}(\Q))=1$ if and only if  one of the following cases happen:
\begin{align*}
\text{{\bf Case 1:}} \quad &\sum_{k=1}^{n}\floor*{\frac{\al_k}{m}}=0 \text{ and }
 n+\sum_{k=1}^{n}\floor*{\frac{\al_k-t_k(i)}{m}}\leq \be(i)\text{ for every } 1\leq i\leq m-1, \text{ or}\\
\text{{\bf Case 2:}}\quad  & \sum_{k=1}^{n}\floor*{\frac{\al_k}{m}}\leq -1, \,\, n+\sum_{k=1}^{n}\floor*{\frac{\al_k-t_k(j)}{m}}-1=\be(j)\text{ for some }1\leq j \leq m-1,\\
&\quad \text{and }\,\, n+\sum_{k=1}^{n}\floor*{\frac{\al_k-t_k(i)}{m}}\leq \be(i)\text{ for every }1\leq i \leq m-1 \text{ with } i\neq j.
\end{align*}
In the first case, since $\sum_{i=1}^{m-1}\floor*{\frac{\al_k-i}{m}}=\al_k-\floor*{\frac{\al_k}{m}}-m+1$ for each $1\leq k \leq n$, we have
\begin{align*}
g(\cX)&=\sum_{i=1}^{m-1}\be(i)
\geq \sum_{i=1}^{m-1}\left(n+\sum_{k=1}^{n}\floor*{\frac{\al_k-t_k(i)}{m}}\right)\\
&=n(m-1)+\sum_{k=1}^{n}\sum_{i=1}^{m-1}\floor*{\frac{\al_k-t_k(i)}{m}}
=n(m-1)+\sum_{k=1}^{n}\sum_{i=1}^{m-1}\floor*{\frac{\al_k-i}{m}}\\
&=n(m-1)+\sum_{k=1}^{n}\left(\al_k- \floor*{\frac{\al_k}{m}}-m+1\right)=\sum_{k=1}^{n}\al_k\\
&=g(\cX),
\end{align*}
and item $i)$ follows. 
In the second case,  we have   
\begin{align*}
g(\cX)&=\sum_{i=1}^{m-1}\be(i)
\geq \sum_{i=1}^{m-1}\left(n+\sum_{k=1}^{n}\floor*{\frac{\al_k-t_k(i)}{m}}\right)-1\\
&=n(m-1)+\sum_{k=1}^{n}\sum_{i=1}^{m-1}\floor*{\frac{\al_k-t_k(i)}{m}}-1
=n(m-1)+\sum_{k=1}^{n}\sum_{i=1}^{m-1}\floor*{\frac{\al_k-i}{m}}-1\\
&=n(m-1)+\sum_{k=1}^{n}\left(\al_k- \floor*{\frac{\al_k}{m}}-m+1\right)-1
=\sum_{k=1}^{n}\al_k-\sum_{k=1}^{n}\floor*{\frac{\al_k}{m}}-1\\
&\geq \sum_{k=1}^{n}\al_k
=g(\cX),
\end{align*}
and item $ii)$ follows.
\end{proof}

\begin{theorem}\label{teo2-car-div-deg-g}
Following the same notation as in Theorem \ref{teo_dimension_RRS_Kummer}, assume that $2\leq n \leq q$. Let $D\in  \Div(\cX)$ be a divisor with $\supp(D)\subseteq\{Q_1, Q_2, \dots, Q_n\}$ and write $D$ in the form
$$D=\sum_{k=1}^{n}(mj_k+\al_k)Q_k, 
\text{ where }  j_k \in \Z \text { and } 0\leq \al_k\leq m-1 \text{  for  }  k= 1, \dots,  n.$$ Then $D$ is a non-special divisor of degree $g(\cX)$ if and only 
if $D$ satisfies one of the following conditions:

\begin{enumerate}[ i)]
	\item\label{teo2-car-div-deg-g-i}  $\sum_{k=1}^{n}j_k=0$  and 
$ n+\displaystyle\sum_{k=1}^{n}\floor*{\frac{\al_k-t_k(i)}{m}}=\be(i) \text{ for every } 1\leq i \leq m-1;$ or

\item\label{teo2-car-div-deg-g-ii} 
$\sum_{k=1}^{n}j_k=-1,$
\begin{align*}
&n-2+\displaystyle\sum_{k=1}^{n}\floor*{\frac{\al_k-t_k(j)}{m}}=\be(j)  \text{ for some } 1\leq j \leq m-1, \text{ and}\\
&n-1+\displaystyle\sum_{k=1}^{n}\floor*{\frac{\al_k-t_k(i)}{m}}=\be(i) \text{ for every } 1\leq i \leq m-1  \text{ with } i\neq j.
\end{align*}
\end{enumerate}
In particular, $D$ is an  non-special effective divisor of degree $g(\cX)$ if and only if 
		\begin{enumerate}[a)]
			\item $D=\sum_{k=1}^{n}\al_kQ_k, \,\,  0\leq \al_k\leq m-1$, and
			\item $
			n+\displaystyle\sum_{k=1}^{n}\floor*{\frac{\al_k-t_k(i)}{m}}=\be(i) \text{ for every } 1\leq i \leq m-1.
			$
		\end{enumerate}
\end{theorem}

\begin{proof}
It is clear that if the divisor $D$ satisfies condition $i)$ or $ii)$, then $D$ satisfies the conditions of Theorem \ref{teo-car-div-deg-g}. Hence, $D$ is a non-special divisor of degree $g(\cX)$. 

Conversely, if $D=\sum_{k=1}^{n}(mj_k+\al_k)Q_k$ is a non-special divisor of degree $g(\cX)$, where $j_k \in \Z$ and $0\leq \al_k\leq m-1$ for $k= 1, \dots,  n$, then $D$ satisfies one of the conditions of Theorem \ref{teo-car-div-deg-g}.

If $D$ satisfies the condition $i)$ of Theorem \ref{teo-car-div-deg-g} then    
$$0=\sum_{k=1}^{n}\floor*{\frac{mj_k+\al_k}{m}}=\sum_{k=1}^{n}\left(j_k+\floor*{\frac{\al_k}{m}}\right)=\sum_{k=1}^{n}j_k,$$
and for every $1\leq i \leq m-1$
\begin{align*}
\be(i)&=n+\sum_{k=1}^{n}\floor*{\frac{mj_k+\al_k-t_k(i)}{m}}
=n+\sum_{k=1}^{n}\left(j_k+\floor*{\frac{\al_k-t_k(i)}{m}}\right)
\\
&=n+\sum_{k=1}^{n}\floor*{\frac{\al_k-t_k(i)}{m}}.	
\end{align*}

In a similar way, if $D$ satisfies the condition $ii)$ of Theorem \ref{teo-car-div-deg-g} then 
we obtain that 
\begin{align*}
&\bullet \textstyle\sum_{k=1}^{n}j_k=-1,  \\
&\bullet n-2+\textstyle\sum_{k=1}^{n}\floor*{\frac{\al_k-t_k(j)}{m}}=\be(j), \text{ and }\\ &\bullet n-1+\textstyle\sum_{k=1}^{n}\floor*{\frac{\al_k-t_k(i)}{m}}=\be(i) \text{ for every } 1\leq i \leq m-1  \text{ with } i\neq j.
\end{align*}
This concludes the first part of the proof. The last statement of the theorem follows immediately from the first part of the proof.
\end{proof}

\section{Explicit non-special divisors of  degree \texorpdfstring{$g-1$}{g-1}}

Aiming to construct LCP of AG codes based on Theorem \ref{thmlcp3.5}, this section provides a more explicit characterization of non-special divisors of small degree on Kummer extensions. One of these characterizations generalizes the construction of non-special divisors of degree $g-1$ given in \cite[Theorem 12]{CLM2024}  for Kummer extensions of the form 
\begin{equation} \label{separable}
y^m=f(x), \text{ where } f(x) \text{ is separable and  } (\deg(f), m)=1. 
\end{equation}
In that work, the authors use a classification of effective non-special divisors of degree $g$ obtained in \cite[Theorem 8]{CLM2024} and apply Lemma \ref{lemma1} to derive non-special divisors of degree $g-1$ of the form $D-P$, where $D$ is an effective non-special divisor of degree $g$ and $P \notin \supp(D)$ a rational place.  In Theorem \ref{teo_g-1_Kummer}  we present all non special divisors of degree $g-1$  for Kummer extensions as in (\ref{separable}), thereby obtaining additional families of such divisors.

We start by fixing some notation. Let $\mathcal{X}$ be the algebraic curve defined in (\ref{equationX}), and let $Q_1, Q_2, \dots, Q_n$ be the totally ramified places of $\mathbb{F}_q(\mathcal{X})/\mathbb{F}_q(x)$.   	 Denote by 
	 $S_n$ the symmetric group of permutations of the set $\{1, 2, \dots, n\}$ and, for $\negalpha= (\alpha_1, \dots, \alpha_n) \in \Z^n$ and $\sigma \in S_n$, define
$$
\sigma(\negalpha):=(\al_{\sigma(1)}, \dots, \al_{\sigma(n)}).
$$

First, we explicitly determine all non-special divisors of degree $g-1$ for Kummer extensions of the form $y^m=f(x)$ with only one place at infinity and $f(x)$ separable. For this, we need first the following technical simple lemma.

\begin{lemma}\label{lemma_floors}
Let $m, n\in \N$ be positive integers such that $n<m$. Then, for $\al, i\in \Z$ such that $0\leq \al, i \leq m-1$, we have
	$$
	\floor*{\frac{\al+(i+1)n}{m}}-\floor*{\frac{\al+in}{m}}=
	\begin{cases}
		1, & \text{if }i=\ceil*{\frac{km-\al}{n}}-1\text{ with }1\leq k\leq n,\\
		0, & \text{otherwise}.
	\end{cases}
	$$
\end{lemma}
\begin{proof}
It is clear that $\lfloor\frac{\al+(i+1)n}{m}\rfloor-\floor*{\frac{\al+in}{m}}$ is equal to $0$ or $1$. 

If
$\lfloor\frac{\al+(i+1)n}{m}\rfloor-\floor*{\frac{\al+in}{m}}=1$, then $$\frac{\al+in}{m}<k:=\floor*{\frac{\al+(i+1)n}{m}}\leq\frac{\al+(i+1)n}{m},$$ where $1\leq k \leq n$ since $0\leq \al, i \leq m-1$. Therefore $i<\frac{km-\al}{n}\leq i+1$ and consequently $i=\ceil*{\frac{km-\al}{n}}-1$.

Conversely, if $i=\ceil*{\frac{km-\al}{n}}-1$ with $1\leq k \leq n$, then $i<\frac{km-\al}{n}\leq i+1$. Therefore $\frac{\al+in}{m}<k\leq\frac{\al+(i+1)n}{m}$. This implies that  $$\floor*{\frac{\al+in}{m}}+1\leq k\leq\floor*{\frac{\al+(i+1)n}{m}},$$ and consequently $\lfloor\frac{\al+(i+1)n}{m}\rfloor\neq\floor*{\frac{\al+in}{m}}$. We conclude  $\lfloor\frac{\al+(i+1)n}{m}\rfloor-\lfloor\frac{\al+in}{m}\rfloor=1$.
\end{proof}

In the next theorem, we explicitly determine all non-special divisors $D$ for algebraic curves $\cY : y^m = f(x)$, where $f(x)$ is a separable polynomial with $(m, \deg(f)) = 1$ and the support of $D$ is contained in the set of places corresponding to zeros and the pole of $f(x)$ in $\mathbb{F}_q(\cY)$. This class of curves often appears in the literature such as the Hermitian curve, Norm-Trace curve \cite{G2003}, and the curve $y^{q^r+1}=x^q+x$ over $\mathbb{F}_{q^{2r}}$ with $r\geq 1$ odd \cite{KKO2001}. 

\begin{theorem}\label{teo_g-1_Kummer}
		Let $\cY$ be the algebraic curve over $\fq$ defined by the affine equation 
		$$\cY: \quad y^m=f(x),$$
		where $m\geq 2$ and $f(x)\in \fq[x]$ is a separable polynomial of degree $2\leq n \leq q$ such that $(m, n)=1$. Let $Q_0$ be the only place at infinity in $\fq(\cY)$, $Q_1, Q_2, \dots, Q_n$ be all rational places in $\fq(\cY)$ corresponding to the zeros of $f(x)$. Define  $\Q=(Q_1, Q_2, \dots, Q_n)$, and  $D\in \Div(\cY)$ a divisor with $\supp(D)\subseteq\{Q_0, Q_1, \dots, Q_n\}$. 
		
		Then $D$ is a non-special divisor of degree $g(\cY)-1$ if and only if \begin{equation}\label{eq0}
			D=\sum_{k=0}^{n}mj_kQ_k+\al_0Q_0+D_{\sigma(\negalpha)}(\Q),
		\end{equation} 
		where $0\leq \al_0 \leq m-1$, $(j_0, \dots, j_n)\in \Z^{n+1}$ is such that $\sum_{k=0}^{n}j_k=-1$, $\sigma\in S_n$, and
		$$
		\negalpha=(\underbrace{0, 0, 0,\dots, 0}_{\floor*{\frac{\al_0+n}{m}}\text{ times}}, \underbrace{1, 1, 1, \dots, 1, 1}_{\floor*{\frac{\al_0+2n}{m}}-\floor*{\frac{\al_0+n}{m}}\text{ times}}, \dots, \underbrace{m-2, m-2, \dots, m-2}_{\floor*{\frac{\al_0+(m-1)n}{m}}-\floor*{\frac{\al_0+(m-2)n}{m}}\text{ times}},  \underbrace{m-1,\dots, m-1}_{\ceil*{\frac{n-\al_0}{m}}\text{ times}}).
		$$	
		
		In particular, if $n<m$ then $D\in \Div(\cY)$ is a non-special divisor of degree $g(\cY)-1$ with $\supp(D)\subseteq\{Q_0, Q_1, \dots, Q_n\}$ if and only if $$
		D=\sum_{k=0}^{n}mj_kQ_k+\al_0Q_0+\sum_{k=1}^{n}\left(\ceil*{\frac{km-\al_0}{n}}-1\right)Q_{\sigma(k)},
		$$
		where $0\leq \al_0\leq m-1$, $(j_0, \dots, j_n)\in \Z^{n+1}$ is such that $\sum_{k=0}^{n}j_k=-1$, and $\sigma\in S_n$.
	\end{theorem}

	\begin{proof}
		 We begin by noting that since $f(x)$ is separable we have $\lambda_0=-n$ and $\lambda_k=1$ for $1\leq k \leq n$.  So, by definition, the following two functions in (\ref{funcoestsbi}) simplify: $t_0(i)= (-in) \mod m$,  $t_k(i)=i$ for $1\leq k\leq n$, and 
			\begin{equation}\label{def}
\beta(i)=\sum_{k=0}^n \ceil*{\frac{i\lambda_k}{m}}-1=\ceil*{\frac{-in}{m}}+n-1=n-1-\floor*{\frac{in}{m}}.
			\end{equation}
Suppose that there exists a non-special divisor $D\in \Div(\cY)$ of degree $g(\cY)-1$ with $\supp(D)\subseteq\{Q_0, Q_1, Q_2, \dots, Q_n\}$. From Theorem \ref{teo-car-div-deg-g-1}, 
		$$
		D=\sum_{k=0}^{n}mj_kQ_k+\sum_{k=0}^{n}\al_kQ_k,
		$$
		where $(j_0, j_1, \dots, j_n)\in \Z^{n+1}$ is such that $\sum_{k=0}^{n}j_k=-1$, $0\leq \al_k\leq m-1$ for every $0\leq k \leq n$, and
		\begin{equation}\label{eq1}
			\be(i)=n+\floor*{\frac{\al_0-t_0(i)}{m}}+\sum_{k=1}^{n}\floor*{\frac{\al_k-i}{m}} \quad \text{for every}\quad 1\leq i \leq m-1.
		\end{equation}
		Now, for each $0\leq j \leq m-1$, define $\gamma_j:=|\{1\leq k\leq n: \al_k=j\}|$. Then $\sum_{i=0}^{m-1}\gamma_i=n$ and
		\begin{align*} 
				n-1-\floor*{\frac{in}{m}}&=n+\floor*{\frac{\al_0-t_0(i)}{m}}+\sum_{k=1}^{n}\floor*{\frac{\al_k-i}{m}}\\
				&=n+\floor*{\frac{\al_0+in}{m}}-\ceil*{\frac{in}{m}}-\sum_{k=1}^{i-1}\gamma_k.
			\end{align*}
			This yields 
		$$
		\gamma_0+\gamma_1+\cdots+\gamma_{i-1}=\floor*{\frac{\al_0+in}{m}}\quad \text{for every }1\leq i \leq m-1.
		$$
		Thus, we obtain the system of equations 
		$$
		\begin{cases}
			\gamma_0=\floor*{\frac{\al_0+n}{m}}\\
			\gamma_0+\gamma_1=\floor*{\frac{\al_0+2n}{m}}\\
			\vdots\\
			\gamma_0+\gamma_1+\cdots+\gamma_{m-2}=\floor*{\frac{\al_0+(m-1)n}{m}}\\
			\gamma_0+\gamma_1+\cdots+\gamma_{m-1}=n
		\end{cases},
		$$
		whose unique solution is $$\gamma_i=\floor*{\frac{\al_0+(i+1)n}{m}}-\floor*{\frac{\al_0+in}{m}}\quad  \text{for} \quad 0\leq i\leq m-1.
		$$
		Therefore, the divisor $D$ has the desired form.
			 
Conversely, if $D$ is defined as in (\ref{eq0}), then $D$ satisfies the conditions given in Theorem \ref{teo-car-div-deg-g-1}.
Thus, $D$ is a non-special divisor of degree $g(\cY)-1$.

		 In particular, if $n< m$ then $D$ has the desired form by Lemma \ref{lemma_floors}.
	\end{proof}

\begin{remark}\label{remark_1}
Following the notation as in Theorem \ref{teo_g-1_Kummer},  in  \cite[Theorem 12]{CLM2024}, the authors determined some non-special divisors of degree $g(\cY)-1$ with support contained in $\{Q_0, Q_1, \dots, Q_n\}$. In Theorem \ref{teo_g-1_Kummer} we determine all such divisors, generalizing this result. In particular, for $\alpha_0=m-1$, $j_0=-1$, and $j_1=j_2=\cdots=j_n=0$ in the previous theorem, we obtain all divisors determined in \cite[Theorem 12]{CLM2024}.

On the other hand, for $\alpha_0=0$ we obtain an explicit description of all non-special divisors of degree $g(\cY)-1$ with support contained in the set $\{Q_1, Q_2, \dots, Q_n\}$ of the zeros of the polynomial $f(x)$. This description remains valid without assuming the condition $(m, n)=1$.
\end{remark}

In the next theorem we explore other applications of constructions of non-special divisors $D$ of degree $g-1$ using results from the last section. We focus in Kummer extensions $\cX: \, y^m=f(x)$, where $f(x)$ is not necessarily separable and the support of  $D$ is contained in a certain set of totally ramified places of $\fq(\cX)/\fq(x)$ with multiplicity satisfying a special condition.

\begin{theorem}\label{carac-grad-g-1}
	Let $\mathcal{X}$ be the algebraic curve defined in (\ref{equationX}) and  $Q_1, Q_2, \dots, Q_n$ be totally ramified places of $\mathbb{F}_q(\mathcal{X})/\mathbb{F}_q(x)$, where $2 \leq n \leq q$. Suppose $\la_s\equiv 1\bmod{m}$ for every $1\leq s \leq n$. Then there exists a non-special divisor $D$ of degree $g(\cX)-1$ with $\supp(D)\subseteq\{Q_1, Q_2, \dots, Q_n\}$ if and only if 
	$$\be(m-1)\leq \cdots \leq \be(2) \leq \be(1) \leq n-1,$$
	and
	\begin{equation}\label{equation_form_D}
	D=mD_{\negj}(\Q)+D_{\sigma(\negalpha)}(\Q),
	\end{equation}
	where
	$\negj=(j_1, \dots, j_n)\in \Z^n$ is such that $\sum_{k=1}^{n}j_k=-1$, $\sigma\in S_n$, and
	$$
	\negalpha=(\underbrace{0, 0,0,\dots,0}_{n-1-\be(1)\text{ times}}, \underbrace{1, 1, 1, \dots, 1}_{\be(1)-\be(2)\text{ times}}, \dots, \underbrace{m-2, \dots, m-2}_{\be(m-2)-\be(m-1)\text{ times}},  \underbrace{m-1,\dots, m-1}_{\be(m-1)+1\text{ times}}) \in \Z^n.
	$$

\end{theorem}
\begin{proof}
	We begin by noting that if $\lambda_s \equiv 1 \bmod{m}$, then $t_s(i) = i$ for every $1 \leq i \leq m-1$.
	Suppose 
	$$\be(m-1)\leq \cdots \leq \be(2) \leq \be(1) \leq n-1$$ 
and $D$ is of the form \eqref{equation_form_D}. Then, $D$  satisfies Theorem \ref{teo-car-div-deg-g-1}, and we conclude it is a non-special divisor of degree $g(\cX)-1$.
	
	Conversely, suppose that there exists a non-special divisor $D\in \Div(\cX)$ of degree $g(\cX)-1$ with $\supp(D)\subseteq\{Q_1, Q_2, \dots, Q_n\}$. From Theorem \ref{teo-car-div-deg-g-1}, 
	$$
	D=mD_{\negj}(\Q)+\sum_{k=1}^{n}\al_kQ_k,
	$$
	where $\negj=(j_1, \dots, j_n)\in \Z^n$ is such that $\sum_{k=1}^{n}j_k=-1$, $0\leq \al_k\leq m-1$ for every $1\leq k \leq n$, and 
	\begin{equation}\label{eq_condition_g-1}
		\be(i)=n-1+\sum_{k=1}^{n}\floor*{\frac{\al_k-i}{m}} \quad \text{for every}\quad 1\leq i \leq m-1.
	\end{equation}
	Now, for each $0\leq j \leq m-1$, define $\gamma_j:=|\{1\leq k\leq n: \al_k=j\}|$. Then $\sum_{i=0}^{m-1}\gamma_i=n$. Moreover, from Equation (\ref{eq_condition_g-1}) we have that, for each $1\leq i \leq m-1$,
	$$
	\be(i)=n-1-\gamma_0-\gamma_1-\cdots-\gamma_{i-1}.
	$$
	Which is equivalent to the following system of equations 
	$$
	\begin{cases}
		\gamma_0=n-1-\be(1)\\
		\gamma_0+\gamma_1=n-1-\be(2)\\
		\vdots\\
		\gamma_0+\gamma_1+\cdots+\gamma_{m-2}=n-1-\be(m-1)\\
		\gamma_0+\gamma_1+\cdots+\gamma_{m-1}=n
	\end{cases},
	$$
	whose unique solution is $\gamma_0=n-1-\be(1)$, $\gamma_i=\be(i)-\be(i+1)$ for every $1\leq i\leq m-2$, and $\gamma_{m-1}=\be(m-1)+1$. Finally, since by definition $0\leq \gamma_{i}$ for every $0\leq i \leq m-1$, we conclude that $\be(m-1)\leq \cdots \leq \be(2) \leq \be(1)\leq n-1$ and $D$ has the desired form.
\end{proof}

We now present a family of curves suitable for applying the previous theorem, followed by an illustrative example.
	\begin{corollary}\label{coresquisita}
Let $\cW$ be the curve over $\fq$ defined by the equation 
$$
\cW:\quad y^m=h_1(x)h_2(x)^\la,$$
where $h_1, h_2\in \fq[x]$ are coprime and separable polynomials of degree $n$ and $s$ respectively,  
and $1\leq \la\leq m-1$. Let $Q_1, Q_2, \dots, Q_{n}$ be the places in $\fq(\cW)$ associated to the zeros of $h_1(x)$. 

If $n\geq s(m-\la)$, then $D\in \Div (\cW)$ is a non-special divisor of degree $g(\cW)-1$ with support contained in $\{Q_1, \dots, Q_n\}$ if and only if $D$ is of the form given in Equation (\ref{equation_form_D}).
	\end{corollary}
	\begin{proof}
	This result is a consequence of Theorem \ref{carac-grad-g-1}. To apply it to the curve $\cW$, it is sufficient to verify that $\be(m-1)\leq \cdots \leq \be(2) \leq \be(1)\leq n-1$. For this, note that for $1\leq i \leq m-2$ we have
	\begin{align*}
	\be(i)-\be(i+1)&=s\left(\ceil*{\frac{i\la}{m}}-\ceil*{\frac{(i+1)\la}{m}}\right)+\floor*{\frac{(i+1)(n+\la s)}{m}}-\floor*{\frac{i(n+\la s)}{m}}\\
	&\geq s\left(\ceil*{\frac{-\la}{m}}-1\right)+\floor*{\frac{n+\la s}{m}}\\
	&=\floor*{\frac{n+\la s}{m}}-s\\
	&=\floor*{\frac{n-s(m-\la)}{m}}\geq 0.
	\end{align*}  
	Analogously, $\be(1)\leq n-1$. Hence, the result follows.
	\end{proof}

	\begin{example}\label{exampleZ}
In this example, we apply Corollary \ref{coresquisita} to a family of maximal curves obtained in \cite[Theorem 3.4]{MQ2022} as a subcover of the $BM$ curve.

Let $r \geq 3$ be odd, $m \geq 2$ a divisor of $\frac{q^r+1}{q+1}$, and $1 \leq d < m$ a divisor of $q-1$. Consider the curve over $\mathbb{F}_{q^{2r}}$ given by the affine equation
$$\mathcal{Z}:\quad y^m = x^{m-d}(1-x^{d(q+1)}).$$
The curve $\mathcal{Z}$ is an $\mathbb{F}_{q^{2r}}$-maximal curve with genus $g(\mathcal{Z})=\frac{d(m-1)(q+1)}{2}$. With notation as in Corollary \ref{coresquisita}, we have that $n=d(q+1)$, $s=1$, and $\la=m-d$. Thus 
$$n-s(m-\la)=d(q+1)-d=dq\geq 0$$ and therefore the conditions of Corollary \ref{coresquisita} are satisfy. This implies that all non-special divisor of degree $g(\mathcal{Z})-1$ with support in the zeros $Q_1, Q_2, \dots, Q_{d(q+1)}$ of $1-x^{d(q+1)}$ are of the form \eqref{equation_form_D}, where 
$$d(q+1)-1-\be(1)=\floor*{\frac{dq}{m}}, \quad \be(m-1)+1=\ceil*{\frac{dq}{m}}+1,$$
and
$$
\be(i)-\be(i+1)=\floor*{\frac{(i+1)d}{m}}-\floor*{\frac{id}{m}}+\floor*{\frac{(i+1)dq}{m}}-\floor*{\frac{idq}{m}} \geq 0  \quad \text{for }1\leq i\leq m-2.
$$

\end{example}

	As another consequence of Theorem \ref{carac-grad-g-1}, for the Hermitian curve $\mathcal{H}_q: y^{q+1}=x^q+x$ over $\fqs$, we obtain a simple explicit expression for all non-special divisors of genus $g(\cH_q)-1$ whose support is contained in the totally ramified places of $\fqs(\cH)/\fqs(x)$.

	\begin{corollary}
		Let $\mathcal{H}_q$ be the Hermitian curve over $\fqs$ defined by the affine equation $y^{q+1}=x^q+x$. Let $Q_0, Q_1, \dots, Q_{q}$ be all the totally ramified places in the extension $\fqs(\mathcal{H}_q)/\fqs(x)$, and $\Q=(Q_0, Q_1, \dots, Q_{q})$. Then $D\in \Div(\mathcal{H}_q)$ is a non-special divisor of degree $g(\mathcal{H}_q)-1$ with $\supp(D)\subseteq \{Q_0, Q_1, \dots, Q_{q}\}$ if and only if
		$$
		D=(q+1)D_{\negj}(\Q)+\sum_{k=0}^{q}kQ_{s_k},
		$$
		where $\negj=(j_0, \dots, j_{q})\in \Z^{q+1}$ is such that $\sum_{k=0}^{q}j_k=-1$ and $Q_{s_1}, \dots, Q_{s_q}\in \{Q_0, \dots, Q_{q}\}$ are pairwise distinct.
	\end{corollary}
	\begin{proof}
	Note that, for the Hermitian curve $\cH_q$, the multiplicity of any totally ramified place of $\fqs(\cH_q)/\fqs(x)$ is congruent to $1$ modulo $q+1$. Moreover, $\be(i)=q-i$ for every $1\leq i \leq q$ and therefore $\be(q)\leq \cdots \leq \be(1) \leq q$. Thus, the result follows from Theorem \ref{carac-grad-g-1}.	
	\end{proof}

\section{Construction of Linear Complementary pairs of AG codes}

Taking advantage of the classification of non-special divisors of degree $g-1$ with support contained in the totally ramified  places  in a Kummer extension as in Theorem \ref{teo-car-div-deg-g-1}, we present general constructions of LCP of AG codes. Some of these constructions generalize the  constructions that have been made in \cite{CLM2024} and \cite{CMQ2025}, in the sense that the non-special divisor of degree $g-1$ is generic.

Through this section we fix the following notation for the Kummer extension that will be used: 
\begin{enumerate}[i)]
\item Let $F:=\fq(\cX)/\fq(x)$ be a Kummer extension of genus $g\geq 1$ defined by the affine equation
\begin{equation}\label{kummer-notseparable} \cX:\quad y^m=f(x)=\prod_{k=1}^{r}(x-a_k)^{\la_k}, \quad  1\leq \la_k\leq m-1 \text{ for every }1\leq k \leq r,
\end{equation}
	where $m, r\geq 2$, $\char(\fq)\nmid m$, $a_1, \dots, a_r\in \fq$ are pairwise distinct, and $\la_0:=\sum_{k=1}^{r}\la_k$ is coprime with $m$. 
	
	\item Let $Q_0$ be the only place at infinity in $\fq(\cX)$ and assume that $Q_1, Q_2, \dots, Q_n$  are totally ramified places in $\fq(\cX)/\fq(x)$ corresponding to the zeros $a_1 \dots, a_n$ of $f(x)$ respectively, where $2 \leq n \leq r$. 
	
	\item  For $1\leq i\leq t$, choose  $P_i$ in $\fq(x)$ a rational place which splits completely in $\fq(\cX)/\fq(x)$, and let $\{R_1,\ldots, R_N\}$ be the set of all $N=mt$ rational places in $\fq(\cX)$ lying over the places $P_1,\ldots, P_t$.
	
	\item \label{equiv} Let $P_b$ be a rational place in  $\fq(x)$ corresponding to $b \in \fq$. If $Q$ is the only place over $P_b$, then  $mQ \sim mQ_0$ since $(x-b)_{\fq(\cX)}=mQ-mQ_0$. If $R_1, \dots, R_m$ are all the places over $P_b$, then $R_1+ \cdots +R_m \sim mQ_0$ since $(x-b)_{\fq(\cX)}=R_1+ \cdots+R_m-mQ_0$.
\end{enumerate}

\begin{theorem}\label{Teocodes1}
	Let $F$ be a Kummer extension. Fix the divisors
 $D=\sum_{k=1}^N R_k$, and $ E=\sum_{k=0}^n \alpha_k Q_k$ a non-special divisor of degree $g-1$. For an integer $s$ such that $\frac{g-1}{\la_0}< s <\frac{N+1-g}{\la_0}$, consider the divisors
\begin{align*}
&G=\sum_{k=1}^n \alpha_k Q_k+(\alpha_0+N-s\la_0)Q_0, \text{ and }\\
 &H=\sum_{k=1}^n (s\la_k+\alpha_k)Q_k+\alpha_0 Q_0+\sum_{k=n+1}^{r}\frac{s\la_k}{(m, \la_k)}\sum_{Q\in \cP_{\fq(\cX)}, Q|P_{a_i}}Q. 
\end{align*}
Then $(C_{\mathcal{L}}(D, G),C_{\mathcal{L}}(D, H))$ is an LCP  of algebraic geometric codes with parameters
$$[N, N-s\la_0] \quad  \text{ and } \quad [N, s\la_0], \text{ respectively}.$$
		\end{theorem}
\begin{proof}
By construction we have  $\supp(D)\cap \supp(G)=\supp(D)\cap \supp(H)=\emptyset$. Define the divisor 
$T=\sum_{k=n+1}^{r}\frac{s\la_k}{(m, \la_k)}\sum_{Q\in \cP_{\fq(\cX)}, Q|P_{a_k}}Q$. 
So, we can writte
  \[ G=E+(N-s\la_0)Q_0 \hspace{3mm} \text{ and } \hspace{3mm}  H=E+\sum_{k=1}^n s\la_kQ_k+T.\] 
Computing the divisor of the function $y^s$ we have  $(y^s)=\sum_{k=1}^n s\la_k Q_k+T-s\la_0Q_0$. This gives 
$$2g-2<\deg(G)=g-1+N-s\la_0<N, \quad 2g-2<\deg(H)=g-1+s\la_0<N,$$ 
and we conclude
\[ \deg(G+H)=N+2g-2. \]

It is easy to check that
$$
\begin{cases}
\gcd(G,H)=E \quad \text{ and }\\ \lmd(G,H)=E+\sum_{k=1}^n s \la_k Q_k +(N-s\la_0)Q_0+T.
 \end{cases}
$$ 
 On the other hand, since $D \sim mtQ_0=NQ_0$ by item \ref{equiv}), we have 
  \[\lmd(G,H)-D \sim  \lmd(G,H)-NQ_0=E+\sum_{k=1}^n s\la_k Q_k + T -s\la_0Q_0=E+(y^s)\sim E.\]
Hence, $(C_{\mathcal{L}}(D, G),C_{\mathcal{L}}(D, H))$ is a LCP of AG codes  by Theorem \ref{thmlcp3.5}.
\end{proof}

\begin{example}
Let $\mathcal{Z}$ be the curve as in Example \ref{exampleZ}. For $r=3$, $q=3$, $d=2$, and $m=7$, we obtain the curve 
$$
\mathcal{Z}: \quad y^7=x^5(1-x^8)
$$
with genus $g(\mathcal{Z})=24$ and 2026 rational places over $\mathbb{F}_{3^6}$. Let $Q_\infty$ be the only place at infinity in $\mathbb{F}_{3^6}(\mathcal{Z})$, $Q_1, \dots, Q_8$ be the places in $\mathbb{F}_{3^6}(\mathcal{Z})$ associated to the zeros of $1-x^8$, and $Q_0$ be the place in $\mathbb{F}_{3^6}(\mathcal{Z})$ associated to the zero of $x$. Therefore, from Example \ref{exampleZ} and Remark \ref{remark_equiv}, 
$$
E=-7Q_\infty+Q_1+2Q_2+3Q_3+3Q_4+4Q_5+5Q_6+6Q_7+6Q_8
$$
is a non-special divisor of degree $g(\mathcal{Z})-1$. For $2\leq s\leq 153$ define the divisors
\begin{align*}
& D=\textstyle\sum_{R\in \mathcal{Z}(\mathbb{F}_{3^6})\setminus\{Q_\infty, Q_1, \dots, Q_8, Q\}}R,\\
& G=(2009-13s)Q_\infty+Q_1+2Q_2+3Q_3+3Q_4+4Q_5+5Q_6+6Q_7+6Q_8, \text{ and}\\
& H=-7Q_\infty+(s+1)Q_1+(s+2)Q_2+(s+3)Q_3+(s+3)Q_4+(s+4)Q_5+(s+5)Q_6\\
&\quad \quad +(s+6)Q_7+(s+6)Q_8+5sQ_0.
\end{align*}
So, from Theorem \ref{Teocodes1}, $(C_{\mathcal{L}}(D, G),C_{\mathcal{L}}(D, H))$ is an LCP of algebraic geometric codes with parameters
$$[2016, 2016-13s] \quad  \text{ and } \quad [2016, 13s].$$
\end{example}

As a direct consequence of the previous theorem, we obtain a construction of LCP codes for the case where $f(x)$ is separable.
\begin{corollary}
With the same notation as in Theorem \ref{Teocodes1}, suppose that $f(x)$ is separable. Fix the divisors $D=\sum_{k=1}^N R_k$, and $E=\sum_{k=0}^n \alpha_k Q_k$ a non-special divisor of degree $g-1$. For $\frac{g-1}{n} < s < \frac{N+1-g}{n}$ consider the divisors
\[
		G=\sum_{k=1}^n \alpha_k Q_k+(\alpha_0+N-sn)Q_0 \hspace{3mm} \text{ and } \hspace{3mm} 
		H=\sum_{k=1}^n (s+\alpha_k) Q_k+\alpha_0Q_0 + \sum_{k=n+1}^r s Q_k. \]
Then $(C_{\mathcal{L}}(D, G),C_{\mathcal{L}}(D, H))$ is an LCP  of algebraic geometry codes with parameters
$$[N, N-sn] \quad  \text{ and } \quad [N, sn].$$
		\end{corollary}
		
Now, we show a construction of LCP codes which generalizes the presented in \cite[Theorem 10]{CMQ2025}.		

\begin{theorem}\label{Teocodes2}
Let $F$ be a Kummer extension of degree $m$ with $f(x)$ separable and $n=r$. Fix the divisors  $D=\sum_{k=1}^N R_k$, and two non-special divisors of degree $g-1$,  $$E_1=\sum_{k=1}^{n} \alpha_k Q_k \, \text{  and  } \, E_2=\sum_{k=0}^{n-1} \beta_k Q_k.$$ Suppose  that $s$ is an integer satisfying
\begin{enumerate}[ i)]
\item \label{Teocodes2i}  $\alpha_k-\beta_k\leq s$ for $1\leq k\leq n-1$,
\item\label{Teocodes2ii}  $\alpha_n\leq s\leq (\beta_0+N)/n$, and 
\item\label{Teocodes2iii} $\frac{g-1+\beta_0-\alpha_n}{n-1} < s < \frac{N-g+1+\beta_0-\alpha_n}{n-1}$.
\end{enumerate} 
Consider the divisors
\[
		G=\sum_{k=1}^{n-1} \alpha_k Q_k+sQ_n+(\beta_0+N-sn)Q_0 \hspace{3mm} \text{ and } \hspace{3mm} 
H=\sum_{k=1}^{n-1} (s+\beta_k) Q_k+\alpha_nQ_n. \]		
Then 		$(C_{\mathcal{L}}(D, G),C_{\mathcal{L}}(D, H))$ is an LCP  of algebraic geometric codes with dimensions $$[N,N-s(n-1)-(\alpha_n -\beta_0)] \, \text{ and } \, [N, s(n-1)+\alpha_n -\beta_0],$$  respectively.
\end{theorem}

\begin{proof}
	
It is clear that $\supp(D)\cap \supp(G)=\supp(D)\cap \supp(H)=\emptyset$. Now, observe that $$\deg (G)=g-1+N+s(1-n)+\beta_0-\alpha_n \hspace{3mm} \text{ and } \hspace{3mm}
\deg (H)=g-1+s(n-1)+\alpha_n-\beta_0,$$ satisfy $2g-2<\deg (G)$, $\deg (H) <N$ and $\deg(G+H)=N+2g-2$. 
On the other hand, by item \ref{equiv}) we have that
$$\gcd(G,H)=\sum_{k=1}^{n-1} \alpha_k Q_k+\alpha_nQ_n =E_1$$
and
\begin{align*}
\lmd(G,H)-D&=\sum_{k=1}^{n-1}(s+\beta_k)Q_k+sQ_n+(\beta_0+N-sn)Q_0-D\\
&=\sum_{k=1}^{n}sQ_k-snQ_0+\sum_{k=0}^{n-1}\beta_kQ_k +NQ_0-D\\
&\sim \sum_{k=0}^{n-1}\beta_kQ_k=E_2
\end{align*}
are non-special divisors of degree $g-1$. Then by Theorem \ref{thmlcp3.5} is $(C_{\mathcal{L}}(D, G),C_{\mathcal{L}}(D, H))$ is an LCP  of algebraic geometric codes.
\end{proof}

We now present a construction of LCPs of codes from non-special divisors of degree $g-1$, and support not necessarily in the totally ramified places of the function field.

\begin{theorem}\label{TeocodesR}
	Let $F$ be a Kummer extension of degree $m$ with $f(x)$ separable and $n=r$. Suppose that  $R_1, \dots R_m$  lying over $P_1$. Consider the divisors
	$$D=\sum_{k=m+1}^N R_k+R_2, \quad \text{ and } \quad  E=\sum_{k=0}^n \alpha_k Q_k$$ a non-special divisor of degree $g$. For an integer $s$ such that $\frac{g-1}{n} < s <\frac{N-m-g+2}{n}$ define the divisors
	\[
	G=\sum_{k=1}^n \alpha_k Q_k+(N-m+\alpha_0-sn)Q_0 \hspace{3mm} \text{ and } \hspace{3mm} 
	H=\sum_{k=1}^n (s+\alpha_k) Q_k+\alpha_0 Q_0-R_1. \]
	
	Then 		$(C_{\mathcal{L}}(D, G),C_{\mathcal{L}}(D, H))$ is an LCP  of algebraic geometric codes with parameters
	$$[N-m+1, N-m+1-sn] \quad  \text{ and } \quad [N-m+1, sn ]$$
\end{theorem}
\begin{proof}
	By construction we have  $\supp(D)\cap \supp(G)=\supp(D)\cap \supp(H)=\emptyset$. Since $\frac{g-1}{n} < s <\frac{N-m-g+2}{n}$, it is easy to check that
	$$
	\begin{cases}
		2g-2 < \deg(G), \deg(H) < N-m+1,\\
		\gcd(G,H)=E-R_1, \text{ and }\\
		 \lmd(G,H)=\sum_{k=1}^n (s+\alpha_k) Q_k +(N-m-sn+\alpha_0)Q_0.
	\end{cases}
	$$ 
	
We also have 
	$$D=\sum_{k=m+1}^N R_k+R_2\sim (N-m)Q_0+R_2 \quad \text{ and } \quad  \sum_{k=1}^n Q_k \sim nQ_0,$$ so 
	$$\lmd(G,H)-D\sim E-R_2,$$
	a non-special divisor of degree $g-1$.
	
	Since $G+H=\gcd(G,H)+\lmd(G,H)$, we obtain	
	\[ \deg(G+H)=\deg(E-R_1)+ \deg(D+E-R_2)=(N-m+1)+2g-2. \] 
	Hence, $(C_{\mathcal{L}}(D, G),C_{\mathcal{L}}(D, H))$ is a LCP of AG codes  by Theorem \ref{thmlcp3.5}.
\end{proof}

%

\bibliographystyle{abbrv}

\bibliography{bib} 

\end{document}